\documentclass[reqno,oneside]{amsart}

\usepackage{amsmath,amssymb,amsthm,amsfonts,enumerate,url,dsfont,colonequals,graphicx,mathrsfs}
\usepackage[a4paper]{geometry}
\usepackage[american]{babel}
\usepackage[usenames]{color}
\usepackage[latin1]{inputenc}
\usepackage[T1]{fontenc}
\usepackage{microtype}
\usepackage{tikz}
\usetikzlibrary{matrix,arrows}

\newtheorem{thm}{Theorem}[section]
\newtheorem{cor}[thm]{Corollary}
\newtheorem{lemma}[thm]{Lemma}
\newtheorem{prop}[thm]{Proposition}

\newtheorem{rem}[thm]{Remark}

\newtheorem{defi}[thm]{Definition}
\theoremstyle{definition}
\newtheorem{exa}[thm]{Example}

\newcommand{\setone}{\mathds{1}}
\newcommand{\coloneqq}{\colonequals}
\newcommand{\eps}{\varepsilon}
\newcommand{\apply}[3][]{\left<#2,#3\right>_{#1}}
\newcommand{\scalar}[3][]{\left(#2|#3\right)_{#1}}
\newcommand{\dx}{\mathrm{d}}
\renewcommand{\phi}{\varphi}
\renewcommand{\rho}{\varrho}

\DeclareMathOperator*{\esssup}{ess\,sup}

\title[Representations of Operators between Bochner spaces]{Properties of Representations of Operators acting between Spaces of Vector-Valued Functions}
\date{March 12, 2009}

\author[D.\ Mugnolo]{Delio Mugnolo}
\address{Delio Mugnolo\\Institute of Analysis\\University of Ulm\\89069 Ulm\\Germany}
\email{delio.mugnolo@uni-ulm.de}

\author[R.\ Nittka]{Robin Nittka}
\address{Robin Nittka\\Institute of Applied Analysis\\University of Ulm\\89069 Ulm\\Germany}
\email{robin.nittka@uni-ulm.de}


\begin{document}
\begin{abstract}
	A well-known result going back to the 1930s states that all bounded
	linear operators mapping scalar-valued $L^1$-spaces into $L^\infty$-spaces are
	kernel operators and that in fact this relation induces an isometric isomorphism
	between the space of such operators and the space of all bounded kernels.
	We extend this result to the case of spaces of vector-valued functions.

	A recent result due to Arendt and Thomaschewski states that the local operators
	acting on $L^p$-spaces of functions with values in separable Banach spaces are precisely
	the multiplication operators.
	We extend this result to non-separable dual spaces.

	Moreover, we relate positivity and other order properties of the
	operators to corresponding properties of the representations.
\end{abstract}
\maketitle

\section{Introduction}
In the theory of partial differential equations one often proves the existence
of solution operators with nice properties without arriving at an explicit formula
to represent them.
In such situations it often helps to have some abstract representation theorems
for linear operators at hand which guarantee that the solution can be expressed
by a formula having a simple structure.
Most prominent among these results are criteria to distinguish operators allowing
a kernel representation, leading to the existence of a so-called Green's function.

The following representation theorem can be traced back at least to
Gelfand~\cite{Gel38} and Kantorovitch and Vulikh~\cite{KanVul37},
see also~\cite[Theorem~2.2.5]{DunPet40}.

\begin{thm}\label{Repr}
	Let $(\Omega_1, \mu_1)$ and $(\Omega_2, \mu_2)$ be $\sigma$-finite measure spaces.
	There is a one-to-one correspondence between the bounded linear operators $T$
	from $L^1(\Omega_1)$ to $L^\infty(\Omega_2)$
	and the bounded kernels $k \in L^\infty(\Omega_1 \times \Omega_2)$.
	More precisely, for every $k \in L^\infty(\Omega_1 \times \Omega_2)$
	\begin{equation}\label{scalarrepform}
		(T_kf) \coloneqq \int_{\Omega_1} k(\omega,\cdot)f(\omega) \, d\mu_1(\omega),\qquad f\in L^1(\Omega_1),
	\end{equation}
	defines a bounded linear operator from $L^1(\Omega_1)$ to $L^\infty(\Omega_2)$.
	Conversely, every bounded linear operator from $L^1(\Omega_1)$ to $L^\infty(\Omega_2)$
	admits such a kernel representation.
\end{thm}

This theorem is not difficult to prove once a few facts about the
projective tensor product
$E \tilde{\otimes}_\pi F$ of two Banach spaces $E$ and $F$ are known.
We refer to~\cite{defflo93} or to the recent monograph~\cite{diefouswa08} for a comprehensive treatment
of Grothendieck's theory of tensor products.
In fact,
\begin{align*}
	\mathscr{L}(L^1(\Omega_1), L^\infty(\Omega_2))
		& \cong (L^1(\Omega_1) \tilde{\otimes}_\pi L^1(\Omega_2))'
		\cong L^1(\Omega_1 \times \Omega_2)' \\
		& \cong L^\infty(\Omega_1 \times \Omega_2),
\end{align*}
where all spaces are isometrically isomorphic. Tracking down the
identifications, we obtain the existence (and uniqueness) of such a kernel.
But see also~\cite[Theorem~1.3]{AreBuk94} for a more elementary proof.

This result can be applied in the context of evolution
equations. For example, it yields that if the operators of a semigroup satisfy
a certain $L^p(\Omega)$-to-$L^q(\Omega)$ estimate, $p < q$, then the individual
operators can be represented by bounded kernels.
Such estimates can be obtained from Sobolev embeddings.
We refer to~\cite[Chapters~2 and~3]{Davies90} and~\cite[\S 7.3]{ArendtSurvey},
for details and applications.

However, sometimes there occur (possibly infinite) systems of partial
differential equations in a quite natural way, see for
example~\cite{Amann01b,DHP07,Kuch05,vBL05} and the references therein.
For such PDEs whose solution families act on vector-valued
function spaces it is necessary to
generalize scalar results like Theorem~\ref{Repr} to operators
acting between spaces of functions taking values in abstract Banach spaces or,
more generally, in Fr\'echet spaces.
One aim of this note is to work out the technical details for the above result.
We also allow for very general measure spaces.

The so-called \emph{weak Dunford-Pettis theorem},
which we present in Section~\ref{repsec},
will be the crucial step in the proof of our main representation theorem.
Section~\ref{mainsec} contains our first main result, the vector-valued
version of Theorem~\ref{Repr}.
Our ideas can also be used to prove a theorem which is similar to a recent
observation due to Arendt and Thomaschewski~\cite{AreTho05}, and we describe
the details in Section~\ref{multops}.
Section~\ref{possec} treats positivity and other order theoretic questions.
Finally, Section~\ref{lpsec} points towards generalizations of our results
into the direction of other $L^p$-spaces.

\section{Representability of Operators}\label{repsec}

In the literature several results have been dubbed the ``Dunford-Pettis
theorem''. The original paper by Dunford and Pettis~\cite{DunPet40}
investigated (among other things) the \emph{representability}
of operators.
\begin{defi}
Let $(\Omega, \mu)$ be a measure space, $F$ a Banach space
and $T$ a bounded linear operator from $L^1(\Omega, \mu)$ to $F$.
We say that $T$ is \emph{representable} if there exists
a \emph{density} $g \in L^\infty(\Omega, F)$ such that 
\[
	Tf = \int_\Omega fg \, \dx\mu\quad\text{for all }f \in L^1(\Omega,\mu)
\]
as a Bochner integral.
\end{defi}

If $F$ is a separable dual space, then by~\cite[Appendix~C5]{defflo93} each
bounded linear operator $T:L^1(\Omega,\mu)\to F'$ is representable given that
$\mu$ is a finite measure.
Closely connected to this is the following property of Banach spaces.

A Banach space $F$ is said to have the \emph{Radon-Nikodym property} if for every
finite measure space $(\Omega, \mu)$ every bounded linear operator
$T$ from $L^1(\Omega,\mu)$ to $F$ is representable.\footnote{In fact, it is
already sufficient to check this condition for $\Omega = (0,1)$ and the
Lebesgue measure, see~\cite[Appendix~D1]{defflo93}.}
We have already mentioned that separable dual spaces have the Radon-Nikodym
property, and so do all reflexive spaces according to the strong Dunford-Pettis
theorem, see~\cite[Appendix~C7]{defflo93}.

However, even if $T\colon L^1(\Omega) \to F'$ fails to be representable,
it is still possible to find a density in a weaker sense. This is
the content of the \emph{weak Dunford-Pettis theorem}, which we are going to state next.
It arises rather naturally in the context of projective tensor products, see~\cite[\S 3.3]{defflo93}.
We start by introducing a function space that will turn out to be precisely the right space
of densities.

\begin{defi}\label{Linfsigmaast}
	Let $F$ be a Banach space and $(\Omega,\mu)$ be a measure space. Denote by
	$\mathcal{L}_{\sigma^\ast}(\Omega; F')$
	the space of $\sigma(F',F)$-measurable, $\sigma(F',F)$-bounded functions from $\Omega$ to $F'$.
	
	Here $f\colon \Omega \to F'$ is called $\sigma(F',F)$-measurable if the scalar
	function $\omega \mapsto \apply{f(\omega)}{v}$ is measurable for every $v \in F$,
	and it is called $\sigma(F',F)$-bounded if
	\[
		\|f\|_{\sigma^\ast} \coloneqq \sup_{v \in B_F} \esssup_{\omega \in \Omega} |\apply{f(\omega)}{v}| < \infty,
	\]
	where $B_F$ is the unit ball in $F$.
	Elements $f$ and $g$ of $\mathcal{L}_{\sigma^\ast}(\Omega;F')$
	are considered equivalent, i.e., $f \sim g$, if $\|f-g\|_{\sigma^\ast} = 0$.

	Finally, we denote by
	$L^\infty_{\sigma^\ast}(\Omega;F') \coloneqq \mathcal{L}_{\sigma^\ast}(\Omega;F') / \sim$
	the space of equivalence classes with respect to this equivalence relation,
	equipped with the norm $\|[f]_\sim \|_\infty \coloneqq \|f\|_{\sigma^\ast}$.
\end{defi}

It is not hard to check that $\|\cdot\|_{\sigma^\ast}$ is a seminorm on
$\mathcal{L}_{\sigma^\ast}(\Omega;F')$ and that $\|\cdot\|_\infty$ is well-defined
and hence a norm on $L^\infty_{\sigma^\ast}(\Omega;F')$.
We point out that the choice of the predual $F$ of $F'$ is implicit in the notation
$L^\infty_{\sigma^\ast}(\Omega;F')$. However, it will always be clear from the
context which predual we consider.

There is an obvious isometric embedding from $L^\infty_{\sigma^\ast}(\Omega;F')$
into $\mathscr{L}(F,L^\infty(\Omega))$, mapping $g$ to the operator $T_g$ defined by
$T_gv \coloneqq \apply{g}{v}$.
In the proof of Theorem~\ref{DunfordPettis} we will show the remarkable fact
that this map is in fact an isomorphism if $(\Omega,\mu)$ is complete and strictly
localizable in the sense of~\cite[\S 211E]{Frem2}.
This is mainly due to the following deep result, which can be found in~\cite[Theorem~IV.3]{Tulcea69},
see also~\cite[\S 341K and 363X(e)]{Frem3}.

\begin{thm}\label{lifting}
	Let $(\Omega,\mu)$ be a complete measure space.
	Then $(\Omega,\mu)$ is strictly localizable if and only if
	there exists a \emph{linear lifting} from $L^\infty(\Omega)$ to
	$\mathcal{M}^\infty(\Omega)$, the space of bounded measurable
	functions on $\Omega$. 
\end{thm}

Recall that a linear lifting is a positive, linear map $\rho$ from
$L^\infty(\Omega;\mathds{R})$ to $\mathcal{M}^\infty(\Omega;\mathds{R})$ such that
$\rho([1]_\sim) = 1$ and $\rho([f]_\sim) \in [f]_\sim$ for every
$[f]_\sim$ in $L^\infty(\Omega)$. This can be extended to a continuous
linear map from $L^\infty(\Omega;\mathds{C})$ to $\mathcal{M}^\infty(\Omega;\mathds{C})$
still respecting the equivalence classes.
We mention that the completeness of the measure space is crucial, see~\cite{Burke93}.

For a Banach space $F$, consider the canonical mapping from $L^\infty_{\sigma^\ast}(\Omega; F')$
to $\mathscr{L}(L^1(\Omega), F')$, assigning to $k$ the unique bounded linear $T_k$ operator satisfying
\begin{equation}\label{tk}
\apply[F',F]{T_kf}{v} = \int_\Omega \apply[F',F]{k(\omega)}{v} f(\omega) \dx \omega
\end{equation}
for all $v \in F$ and $f \in L^1(\Omega)$.

The following \emph{weak Dunford-Pettis theorem} asserts that this mapping
is in fact surjective for complete, strictly localizable measure spaces.\footnote{In
fact, the existence of a linear lifting can even be characterized
by the validity of the weak Dunford-Pettis theorem, cf.~\cite[\S VII.2]{Tulcea69}.}
Our version of this theorem is slightly more general than what can usually be
found in the literature. In fact, it describes precisely which kernels have to
be identified and how this quotient space has to be endowed with a Banach space
norm in order to make the identification of kernels and kernel operators
an isometric isomorphism.

\begin{thm}[weak Dunford-Pettis theorem]\label{DunfordPettis}
	Let $(\Omega, \mu)$ be a complete, strictly localizable measure space,
	and let $F$ be a Banach space.
	Then the canonical mapping $k \mapsto T_k$ is an isometric isomorphism from
	$L^\infty_{\sigma^\ast}(\Omega; F')$ to $\mathscr{L}(L^1(\Omega), F')$.
\end{thm}

\begin{proof}
	Every strictly localizable space is localizable~\cite[\S 211L]{Frem2}, hence
	$L^1(\Omega)' = L^\infty(\Omega)$ under the usual identification~\cite[\S 243G]{Frem2}.
	Thus
	\[
		\mathscr{L}(L^1(\Omega),F') \cong (L^1(\Omega) \tilde{\otimes}_{\pi} F)'
			\cong \mathscr{L}(F,L^\infty(\Omega)),
	\]
	see also~\eqref{commut_assoc} and Lemma~\ref{tensordual}. Consequently, it suffices to show that
	the canonical map from $L^\infty_{\sigma^\ast}(\Omega;F')$ to $\mathscr{L}(F,L^\infty(\Omega))$
	is surjective. In fact, the resulting isometric isomorphism from $L^\infty_{\sigma^\ast}(\Omega;F')$
	to $\mathscr{L}(L^1(\Omega),F')$ is easily checked to be the one described
	in the theorem.

	Let $T \in \mathscr{L}(F,L^\infty(\Omega))$ and define
	$g_v \coloneqq \rho(Tv) \in \mathcal{M}^\infty(\Omega)$,
	where $\rho$ is a linear lifting from $L^\infty(\Omega)$ to $\mathcal{M}^\infty(\Omega)$,
	which exists by Theorem~\ref{lifting}.
	Then $g_v(\omega)$ is linear in $v$ and
	and $|g_v(\omega)| \le \|T\| \left\|v\right\|$
	because $\rho(f)(\omega) \le \esssup_{\omega \in \Omega} |f(\omega)|$
	for every $f \in L^\infty(\Omega)$.

	Thus $\apply{g(\omega)}{v} \coloneqq g_v(\omega)$ defines
	an element $g(\omega) \in F'$ for every $\omega \in \Omega$,
	and $\|g(\omega)\|_{F'} \le \|T\|$. By definition,
	$g$ is $\sigma(F',F)$-measurable and $\|g\|_{\sigma^\ast} \le \|T\| < \infty$,
	hence $g \in \mathcal{L}_{\sigma^\ast}(\Omega;F')$.
	By construction, $T_{\tilde{g}} = T$, where
	$\tilde{g} \coloneqq [g]_\sim \in L^\infty_{\sigma^\ast}(\Omega;F')$
	and $T_{\tilde{g}}$ is defined as in~\eqref{tk}.
	This finishes the proof.
\end{proof}

\begin{rem}
	The theorem states in particular that every bounded linear operator $T$ from $L^1(\Omega)$ into
	a dual space has a $\sigma(F',F)$-density $k \in L^\infty_{\sigma^\ast}(\Omega;F')$, i.e., 
	\[
		Tf = \int_\Omega f k \, \dx\mu\qquad\hbox{ for all }f\in L^1(\Omega)
	\]
	as a $\sigma(F',F)$-Pettis integral. This result can also be found
	in~\cite[Theorem~2.1.6]{DunPet40} and~\cite[3.3]{defflo93}.
\end{rem}

If $F$ is separable, then $\|f(\cdot)\|_{F'}$ is measurable and 
\begin{equation}\label{badnorm}
	\|f\|_\infty = \esssup_{\omega \in \Omega} \|f(\omega)\|_{F'}.
\end{equation}
In fact, $\|f(\omega)\|_{F'} = \sup_n |\apply{f(\omega)}{v_n}|$
for every $\omega \in \Omega$, where $(v_n)_{n\in\mathds{N}}$ is dense in $F$.
This is the norm considered in~\cite[Definition~2.1]{AreTho05}.
But in the general (non-separable) case this norm is not even
well-defined on $\mathcal{L}_{\sigma^\ast}(\Omega; F')$ as the following
example shows.

\begin{exa}\label{nonorm}
	Let $I$ be a set of the same cardinality as the continuum and consider the non-separable
	Hilbert space $F \coloneqq \ell^2(I)$. Then there exists a bounded, $\sigma(F',F)$-measurable
	function $f$ on $\Omega \coloneqq (0,1)$ taking values in $F'$ such that the equivalence
	class of $f$ in $L^\infty_{\sigma^\ast}(\Omega;F')$
	is the zero function, but $\|f(t)\|_{F'} = 1$ for every $t \in \Omega$.
	So if $\|f\|_\infty$ was defined as in~\eqref{badnorm}, its value would depend on
	the choice of the representative.

	In fact, let $(e_t)_{t \in \Omega}$ be an orthonormal basis of $F$.
	Define $\apply{f(t)}{v} \coloneqq \scalar{v}{e_t}$ for $t\in I$, $v\in F$.
	Then $f$ has the properties described above, since
	$\|f(t)\|_{F'} = 1$ for every $t \in \Omega$,
	and $f$ is $\sigma(F',F)$-measurable because for every $v \in F$
	the function $\apply{f(\cdot)}{v}$ vanishes outside a countable
	subset of $\Omega$. More precisely this shows that $f$ is
	equivalent to the zero function.
\end{exa}

\section{Representability of Operators on Spaces of Vector-Valued Functions}\label{mainsec}

We start by giving a definition of the projective tensor product in terms of
a universal property. This approach is not the usual one.
We introduce the tensor product in this way, though, because we will
encounter a similar object in Section~\ref{possec} and want out the similarity
between the definitions.
We refer to~\cite{defflo93} for a systematic introduction to tensor products.
There one also can find proofs of all the mentioned results about the
projective tensor product.

Let $E$ and $F$ be Banach spaces. A pair $(H, \chi)$, where $H$ is a Banach space
and $\phi$ is a bilinear map from $E \times F$ to $H$ such that $\|\chi(u,v)\| = \|u\|\,\|v\|$,
is called a \emph{projective tensor product of $E$ and $F$} if for every vector space
$G$ and every bounded bilinear map $\phi$ from $E \times F$
to $G$ there exists a unique bounded linear operator $T$ from $H$ to $G$ such that
the diagram
\begin{center}
\begin{tikzpicture}
	\centering
	\matrix (m) [matrix of math nodes, row sep=2.5em, column sep=4em, text height=1.5ex, text depth=0.25ex]
	{ E \times F & G \\ H & \\ };
	\path[->]
	(m-1-1) edge node[auto] {$\phi$} (m-1-2)
			edge node[auto,swap] {$\chi$} (m-2-1);
	\path[->,dashed]
	(m-2-1) edge node[auto,swap] {$T$} (m-1-2);
\end{tikzpicture}
\end{center}
commutes, and, moreover, this unique map $T$ satisfies
$\|T\| = \|\phi\|$. Here, as usual, $\|\phi\|$ denotes the least
constant $c$ such that $|\phi(u,v)| \le c \|u\|\,\|v\|$ holds
for all $u \in E$ and $v \in F$.

For any pair of Banach spaces there exists a projective tensor product.
It is standard to check that the projective tensor product is
essentially unique, i.e., for any two projective tensor products
$(H_1,\chi_1)$ and $(H_2,\chi_2)$ there is an isometric isomorphism $J$
from $H_1$ to $H_2$ such that $J \circ \chi_1 = \chi_2$.
We denote this space by $E \tilde{\otimes}_\pi F$ and write
$u \otimes v$ for $\chi(u,v)$. Moreover, $E \otimes F$ denotes
the span of the image of $\chi$, which is dense in $E \tilde{\otimes}_\pi F$.

It follows that for Banach spaces $E$, $F$ and $G$,
\begin{equation}\label{commut_assoc}
	E \tilde{\otimes}_\pi F \cong F \tilde{\otimes}_\pi E 
\quad \text{and} \quad
	(E \tilde{\otimes}_\pi F) \tilde{\otimes}_\pi G
		\cong E \tilde{\otimes}_\pi (F \tilde{\otimes}_\pi G)
\end{equation}
via the natural mappings.

If $(\Omega,\mu)$ is a measure space, we write $L^1(\Omega;E)$ for the Bochner space of measurable,
integrable functions on $\Omega$ taking values in $E$. Then
\begin{equation}\label{L1t}
	L^1(\Omega_1) \tilde{\otimes}_\pi L^1(\Omega_2) \cong L^1(\Omega_1 \times \Omega_2),
\end{equation}
for arbitrary measure spaces $\Omega_1$ and $\Omega_2$, where $\Omega_1 \times \Omega_2$
is equipped with the complete locally determined product measure~\cite[\S 251F]{Frem2},
see~\cite[Exercise~3.27]{defflo93} or~\cite[Notes to \S 253]{Frem2}.
For $\sigma$-finite measure spaces the complete locally determined  product measure
is the completion of the usual (unique) product measure of $\sigma$-finite
measure spaces, see~\cite[\S 251K]{Frem2}.

We recall two further results about tensor products, which can be found for example
in~\cite[III.6.5, III.6.2]{Schaefer99}.

\begin{lemma}\label{L1E}
	Let $(\Omega,\mu)$ be a measure space. Then
	$T\colon L^1(\Omega) \otimes E \to L^1(\Omega;E)$
	defined via 
	\[
		T(f \otimes \eta) \coloneqq f(\cdot)\eta
	\]
	extends to an isometric isomorphism from
	$L^1(\Omega)\tilde{\otimes}_\pi E$ onto $L^1(\Omega;E)$.
\end{lemma}

\begin{lemma}\label{tensordual}
	The dual space of $E \tilde{\otimes}_\pi F$ is $\mathscr{L}(E,F')$,
	where the duality acts as
	\[\apply[\mathscr{L}(E,F'), E \tilde{\otimes}_\pi F]{T}{x \otimes y} \coloneqq \apply[F',F]{Tx}{y}\]
	for $T \in \mathscr{L}(E, F')$ and $x \otimes y \in E \otimes F$.
\end{lemma}

Now we are prepared to prove the vector-valued version of Theorem~\ref{Repr}.
For $E = F = \mathds{R}$, it essentially reduces to Theorem~\ref{Repr}.
But for Banach spaces, in particular non-separable ones, it seems to be new.
Moreover, the measure theoretical assumption, namely the existence of a lifting,
seems to be optimal.

Let $(\Omega_1,\mu_1)$ and $(\Omega_2,\mu_2)$ be measure spaces.
Given $k\in L^\infty_{\sigma^\ast}(\Omega_1 \times \Omega_2; \mathscr{L}(E,F'))$,
we denote by $T_k$ the unique bounded linear operator from $L^1(\Omega;E)$ to
$L^\infty_{\sigma^\ast}(\Omega_2;F')$ that satisfies
\begin{equation}\label{repr:formula}
	\apply[F',F]{(T_kf)(\omega_2)}{v} = \int_{\Omega_1} \apply[F',F]{k(\omega_1,\omega_2)f(\omega_1)}{v} \dx\mu_1(\omega_1)
\end{equation}
for every $v \in F$, $f \in L^1(\Omega_1)$, and $\mu_2$-almost every $\omega_2 \in \Omega_2$.
As before, we equip $\Omega_1 \times \Omega_2$ with the complete locally determined  product measure.

\begin{thm}\label{repr}
	Let $(\Omega_1,\mu_1)$ and $(\Omega_2,\mu_2)$ be complete, strictly localizable
	measure spaces.
	The mapping $k \mapsto T_k$ from $L^\infty_{\sigma^\ast}(\Omega_1 \times \Omega_2; \mathscr{L}(E,F'))$
	to $\mathscr{L}(L^1(\Omega_1; E), L^\infty_{\sigma^\ast}(\Omega_2; F'))$
	is an isometric isomorphism.
\end{thm}

Examples of strictly localizable measure spaces include $\sigma$-finite measure spaces~\cite[\S 211L]{Frem2}
and locally compact spaces equipped with a Radon measure~\cite[\S 416B]{Frem41}.

\begin{proof}
	It is known that the complete locally determined product measure space $\Omega_1 \times \Omega_2$
	is again complete~\cite[\S 251I]{Frem2} and strictly localizable~\cite[\S 251N]{Frem2}.
	Using (in this order)
	\begin{itemize}
		\item Theorem~\ref{DunfordPettis}, 
		\item Lemma~\ref{tensordual} (applied twice),
		\item Lemma~\ref{L1E}, \eqref{commut_assoc},
		\item Lemma~\ref{tensordual},
		\item Lemma~\ref{tensordual} together with Theorem~\ref{DunfordPettis},
	\end{itemize}
	we see that the chain
	\begin{align*}
		\mathscr{L}(L^1(\Omega_1; E), L^\infty_{\sigma^\ast}(\Omega_2; F'))
			& \cong \mathscr{L}(L^1(\Omega_1; E), \mathscr{L}(L^1(\Omega_2), F')) \\
			& \cong (L^1(\Omega_1; E) \tilde{\otimes}_\pi (L^1(\Omega_2) \tilde{\otimes}_\pi F))' \\
			& \cong (L^1(\Omega_1) \tilde{\otimes}_\pi L^1(\Omega_2) \tilde{\otimes}_\pi E \tilde{\otimes}_\pi F)' \\
			& \cong \mathscr{L}(L^1(\Omega_1 \times \Omega_2), (E \tilde{\otimes}_\pi F)') \\
			& \cong L^\infty_{\sigma^\ast}(\Omega_1 \times \Omega_2; \mathscr{L}(E, F'))
	\end{align*}
	consists of isometric isomorphies.
	It is an easy, yet tiresome, exercise to check that the composition of the
	above isometries is in fact the operator given by~\eqref{repr:formula}.
\end{proof}

\begin{rem}
	The mapping described in Theorem~\ref{repr} is an isomorphism even if $E$
	and $F$ are merely Fr\'echet spaces. The proof of this fact is essentially
	the same as the above. One only has to note that all results we had cited in
	the proof remain valid in this case, see~\cite{Flo74} for Theorem~\ref{DunfordPettis},
	and the aforementioned sections of~\cite{Schaefer99} for Lemmata~\ref{L1E} and~\ref{tensordual}.
\end{rem}

In some situations Theorem~\ref{repr} remains true if we replace $L^\infty_{\sigma^\ast}$
by $L^\infty$, i.e., if we consider (strongly) measurable functions only, as the
following corollary shows.

\begin{cor}\label{repr:cor}
	If $(\Omega_1,\mu_1)$ and $(\Omega_2,\mu_2)$ are complete,
	finite measure spaces and $\mathscr{L}(E, F')$ has the Radon-Nikodym property,
	then $L^\infty(\Omega_1 \times \Omega_2; \mathscr{L}(E,F'))$
	is isometrically isomorphic to
	$\mathscr{L}(L^1(\Omega_1, E), L^\infty(\Omega_2, F'))$
	by the same identification as in Theorem~\ref{repr}.
	Moreover,
	\[
		(T_kf)(\omega_2) = \int_{\Omega_1} k(\omega_1,\omega_2)f(\omega_1) \dx\mu_1(\omega_1)
	\]
	as a Bochner integral in $F'$ for $\mu_2$-almost every $\omega_2 \in \Omega_2$.
\end{cor}

\begin{proof}
	We can assume $E \neq \{0\}$ without loss of generality.
	Then $F'$ has the Radon-Nikodym property, too, since it can
	be identified with a closed subspace of $\mathscr{L}(E, F')$.
	Now the claim follows from the following observation.
	If $G$ is a Banach space, $G'$ has the Radon-Nikodym property,
	and $(\Omega,\mu)$ is a finite measure space,
	then the natural embedding of $L^\infty(\Omega;G')$
	into $L^\infty_{\sigma^\ast}(\Omega;G')$ is in fact an
	isometric isomorphism. This follows from Theorem~\ref{DunfordPettis}
	and the definition of the Radon-Nikodym property.
\end{proof}

\begin{rem}
	There is a class of spaces for which the assumptions of the preceding corollary are fulfilled.
	In fact, by a result due to Andrews~\cite{and83},
	$\mathscr{L}(E,F')$ has the Radon-Nikodym property
	whenever $E'$ and $F'$ have the Radon-Nikodym property and, moreover,
	every bounded linear operator from $E$ to $F'$ is compact.
	By Pitt's theorem~\cite[42.3.(10)]{Koethe79} this is the case
	for $E = \ell^p$ and $F = \ell^q$, where
	$1 \le q' < p < \infty$, $\frac{1}{q} + \frac{1}{q'} = 1$.
	There are also some other classes of Banach spaces that
	enjoy these properties (so that Corollary~\ref{repr:cor} can
	be applied), including some
	$L^p$-spaces~\cite{Ros69}, Orlicz sequence spaces~\cite[p.~149]{LinTza77},
	and Lorentz sequence spaces~\cite{AO97}.
\end{rem}

\section{Multiplication Operators}\label{multops}
In this section we apply our techniques to obtain a result that extends
Theorem~2.3 in~\cite{AreTho05} to non-separable dual spaces.
Let $E$ be a Banach space,
$(\Omega,\Sigma,\mu)$ a strictly localizable, complete measure space,
and $p \in [1,\infty]$,
and consider the space
\[
	L^\infty(\Omega; \mathscr{L}_s(E)) \coloneqq
		\{ M\colon \Omega \to \mathscr{L}(E) : M(\cdot)x \in L^\infty(\Omega;E) \text{ for all } x \in E \}.
\]
For every $M \in L^\infty(\Omega; \mathscr{L}_s(E))$,
$(\mathcal{M}_Mf)(\omega) \coloneqq M(\omega)f(\omega)$ defines a
bounded linear operator $\mathcal{M}_M \in \mathscr{L}(L^p(\Omega))$,
even if $E$ is not separable. These operators are called
\emph{multiplication operators on $L^p(\Omega;E)$}.
We remark that for non-separable $E$ the space $L^\infty(\Omega; \mathscr{L}_s(E))$
cannot be normed in the same way as in~\cite{AreTho05}, compare Example~\ref{nonorm}.

We call an operator $T \in \mathscr{L}(L^p(\Omega;E))$ \emph{local}
if for all $f \in L^p(\Omega;E)$ we have
\[
	Tf = 0 \text{ almost everywhere on } \{ \omega : f(\omega) = 0 \}.
\]
This is equivalent to the condition that
\[
	T(\setone_A f) = \setone_A Tf \quad \text{ for all } f \in L^p(\Omega; E) \text{ and all } A \in \Sigma.
\]
It is obvious that multiplication operators are local.
In~\cite{AreTho05} the authors proved the following theorem.
\begin{thm}[Arendt, Thomaschewski]\label{multAT}
	Let $(\Omega,\Sigma,\mu)$ be $\sigma$-finite and $E$ be separable.
	Then for every local operator $T \in \mathscr{L}(L^p(\Omega;E))$
	there exists $M \in L^\infty(\Omega; \mathscr{L}_s(E))$ such
	that $\mathcal{M}_M = T$.
\end{thm}

The proof of Arendt and Thomaschewski could easily be adopted to cover arbitrary strictly
localizable measure spaces. But it heavily exploits the separability of
$E$ to construct a kind of vector-valued linear lifting.
We modify their approach to obtain a similar result for non-separable \emph{dual} spaces.

To this end, we have to consider a different space of multiplicators.
Let $M \in L^\infty_{\sigma^\ast}(\Omega; \mathscr{L}(F'))$.
This means in particular that $\apply{M(\cdot)}{v' \otimes v} = \apply{M(\cdot)v'}{v}$
is measurable for all $v' \in F'$ and $v \in F$, compare
Definition~\ref{Linfsigmaast} and Lemma~\ref{tensordual}.
In this situation we define $(\mathcal{M}_Mf)(\omega) \coloneqq M(\omega)f(\omega)$
and obtain that $Mf\colon \Omega \to F'$ is weakly measurable. It is not
difficult to check that
\[
	\int_\Omega \bigl| \apply{(\mathcal{M}_Mf)(\omega)}{v} \bigr|^p \dx\mu(\omega)
		\le \|M\|_\infty^p \: \|v\|^p \: \|f\|_p^p
\]
for simple functions $f \in L^p(\Omega; F')$, hence
$\apply{\mathcal{M}_Mf}{v} \in L^p(\Omega)$ and $\|\apply{\mathcal{M}_Mf}{v}\|_p \le \|v\| \: \|f\|$
for all $f \in L^p(\Omega;F')$ and all $v \in F$.
Thus $\mathcal{M}_Mf \in L^p(\Omega;F')$ if and only if $\mathcal{M}_Mf$ is
almost separably valued, cf.~\cite[Corollary~1.1.3]{ABHN01}.
In this case, $\mathcal{M}_M$ is a bounded operator on $L^p(\Omega;F')$,
and $\|\mathcal{M}_M\| \le \|M\|_\infty$.

One can easily come up with examples that $\mathcal{M}_Mf \not\in L^p(\Omega;F')$
if $F'$ and $(\Omega,\Sigma,\mu)$ are rich enough, for example using~\cite[Example~1.1.5]{ABHN01}.
Thus not every multiplication operator with multiplier in $L^\infty_{\sigma^\ast}(\Omega; \mathscr{L}(F'))$
defines a local operator. But still, every local operator is a multiplication operator in this sense
as the following theorem shows.

\begin{thm}\label{multMN}
	Let $F$ be a Banach space. Then for every local operator $T \in \mathscr{L}(L^p(\Omega;F'))$
	there exists $M \in L^\infty_{\sigma^\ast}(\Omega; \mathscr{L}(F'))$ such that $\mathcal{M}_M = T$.
\end{thm}

\begin{proof}
	Since $(\Omega,\Sigma,\mu)$ is strictly localizable, there exists a lifting
	$\rho\colon L^\infty(\Omega) \to \mathcal{M}^\infty(\Omega)$ and
	a family $(A_\alpha)_{\alpha \in I}$ of mutually disjoint sets of finite measure in $\Sigma$,
	covering $\Omega$, such that $B \subset \Omega$ is measurable if (and only if) $B \cap A_\alpha$ is measurable
	for all $\alpha$, and $\mu(B) = \sum_\alpha \mu(B \cap A_\alpha)$.

	Let $T \in \mathscr{L}(L^p(\Omega;F')$ be local. Fix $\alpha \in I$ and $v' \in F'$
	and define $M_{\alpha,v'} \coloneqq T(\setone_{A_\alpha} v')$, which lies in $L^p(\Omega;F')$.
	We show that even $M_{\alpha,v'} \in L^\infty(\Omega;F')$. In fact,
	let $\eps > 0$ and define
	\[
		B_\eps \coloneqq \{ \omega \in A_\alpha : \|M_{\alpha,v'}(\omega)\| \ge (\|T\| + \eps) \|v'\| \}.
	\]
	Since $T$ is local,
	\begin{align*}
		\mu(B_\eps) (\|T\| + \eps)^p \|v'\|^p
			& \le \int_{A_\alpha} \setone_{B_\eps} \|M_{\alpha,v'}(\omega)\|^p \dx\mu(\omega)
			= \int_{A_\alpha} \|T(\setone_{B_\eps} v')(\omega)\|^p \dx\mu(\omega) \\
			& \le \|T\|^p \|\setone_{B_\eps} v'\|_p^p
			= \|T\|^p \mu(B_\eps) \|v'\|^p.
	\end{align*}
	This is only possible if $v' = 0$ or $\mu(B_\eps) = 0$.
	In both cases we obtain
	$\|M_{\alpha,v'}(\omega)\| \le \|T\| \: \|v'\|$ for almost
	all $\omega \in \Omega$ by letting $\eps$ tend to $0$.

	Now define $M_{\alpha,v',v} \coloneqq \rho \circ \apply{M_{\alpha,v'}(\cdot)}{v} \in \mathcal{M}^\infty(\Omega)$.
	Then $M_{\alpha,v',v}(\omega)$ is a bilinear function of $v$ and $v'$ and
	$|M_{\alpha,v',v}(\omega)| \le \|T\| \: \|v'\| \: \|v\|$. Hence
	$\apply{M_\alpha(\omega)v'}{v} \coloneqq M_{\alpha,v',v}$
	defines an element $M_\alpha(\omega) \in \mathscr{L}(F')$ such that $\|M_\alpha(\omega)\| \le \|T\|$
	for every $\omega \in \Omega$.
	Define $M(\omega) \coloneqq M_\alpha(\omega)$ for $\omega \in A_\alpha$.
	Then $M\colon \Omega \to \mathscr{L}(F')$ is well-defined, bounded, and
	$\apply{M(\cdot)v'}{v} = \apply{M(\cdot)}{v' \otimes v}$
	is measurable for all $v$ thanks to the properties of $A_\alpha$.

	Since the simple tensors are total in $F' \tilde{\otimes}_\pi F$,
	this shows that $M$ is in $\mathcal{L}_{\sigma^\ast}(\Omega; \mathscr{L}(F'))$
	and can thus be considered to be an element of $L^\infty_{\sigma^\ast}(\Omega;F')$.
	By construction and because $T$ and $\mathcal{M}_M$ are local operators,
	$\mathcal{M}_Mf = Tf$ for all simple functions.
	Approximating arbitrary $f \in L^p(\Omega;F')$ by simple functions,
	this implies that $\mathcal{M}_Mf \in L^p(\Omega;F')$ and $Tf = \mathcal{M}_Mf$
	for all $f \in L^p(\Omega;F')$.
\end{proof}

Finally, we show that Theorems~\ref{multAT} and~\ref{multMN} give the
same result in the cases in which both of them apply.
\begin{prop}
	Let $E = F'$ be a separable dual space. Then
	$L^\infty_{\sigma^\ast}(\Omega; \mathscr{L}(F')) = L^\infty(\Omega; \mathscr{L}_s(F'))$.
\end{prop}
\begin{proof}
	A function $f$ is in $L^\infty_{\sigma^\ast}(\Omega; \mathscr{L}(F'))$
	if and only if $\|f(\cdot)\|$ is essentially bounded and the function
	$\apply{f(\cdot)}{v' \otimes v} = \apply{f(\cdot)v'}{v}$
	is measurable for all $v' \otimes v \in F' \tilde{\otimes}_\pi F$.
	Here we used that $F$ is separable and the simple tensors are total
	in $F' \tilde{\otimes}_\pi F$, compare also the discussion regarding~\eqref{badnorm}.
	By~\cite[Corollary~1.1.3]{ABHN01} this is the case
	if and only if $f(\cdot)v' \in L^\infty(\Omega;F')$ for all $v' \in F'$,
	i.e., $f \in L^\infty(\Omega;\mathscr{L}_s(F')$.
\end{proof}

\section{Positive Representations and Regular Representations}\label{possec}
In this section, we only consider real Banach lattices
and complete, strictly localizable measure spaces.
We denote the positive cone of a Banach lattice $E$ by $E_+$,
and we write $E'_+$ for the set of all positive linear functionals
on $E$.

By definition, a measurable function $f\colon \Omega \to E$ is
positive if $f(\omega) \in E_+$ for almost every $\omega \in \Omega$.
Note that this ordering makes $L^1(\Omega;E)$ into a Banach lattice
such that $|f|(\omega) = |f(\omega)|$ for almost every $\omega \in \Omega$.

For a function $f$ in $L^\infty_{\sigma^\ast}(\Omega;F')$ we say that $f$
is positive if there exists a representative of $f$ taking values only in $F'_+$,
or, which is the same thanks to the positivity of the linear lifting,
$\apply{k(\cdot)}{v} \in L^\infty_+(\Omega)$ for every $v \in F_+$.
This makes $L^\infty_{\sigma^\ast}(\Omega;F')$ isometrically order isomorphic
to $L^1(\Omega;F)'$ via the natural identification, see also the proof
of Theorem~\ref{DunfordPettis}. In particular, $L^\infty_{\sigma^\ast}(\Omega;F')$
is a Banach lattice.
On the subspace $L^\infty(\Omega;F')$ of $L^\infty_{\sigma^\ast}(\Omega;F')$,
this ordering coincides with the pointwise order.

The following two theorems relate the positivity of the operator
to the positivity of the representing kernel and multiplier as
defined in the previous sections, respectively.
In other words, these theorems say that, in the case of Banach lattices,
the identifications of Theorems~\ref{repr} and~\ref{multMN} are isometric lattice
isomorphisms~\cite[Theorem~2.15]{Ali06}.

\begin{thm}\label{reprpos}
	Under the assumptions of Theorem~\ref{repr}, let in addition $E$ and $F$ be Banach lattices.
	Then $k$ is positive if and only if $T_k$ is positive.
\end{thm}
\begin{proof}
	If $k$ is positive, then the right hand side of~\eqref{repr:formula}
	is non-negative whenever $f \in L^1_+(\Omega; E)$ and
	$v \in F_+$. This means $T_kf \ge 0$ almost everywhere,
	which is precisely the meaning of $T_k \ge 0$.

	If, on the other hand, $T_k \ge 0$, then the right hand
	side of~\eqref{repr:formula} is non-negative for every
	$f \in L^1_+(\Omega; E)$ and $v \in F_+$.
	Putting $f \coloneqq u \setone_A$ for an
	arbitrary measurable set $A$ and some $u \in E_+$, we see
	$\apply{k(\cdot,\cdot)u}{v} \ge 0$ almost everywhere.
	Applying a linear lifting
	we obtain a representative of $k$ taking values in the
	positive operators. Hence $k \ge 0$ by definition.
\end{proof}

\begin{thm}\label{multpos}
	Under the assumptions of Theorem~\ref{multMN}, let in addition $F$ be a Banach lattice.
	Then $T$ is positive if and only if $M$ is positive.
\end{thm}
\begin{proof}
	It is obvious that $\mathcal{M}_M \ge 0$ whenever $M \ge 0$.
	
	Now assume that $T = \mathcal{M}_M \ge 0$. We use the notation
	of the proof of Theorem~\ref{multMN}. Note that for $v' \ge 0$
	we have $M_{\alpha,v'} \ge 0$ almost everywhere and hence
	$M_{\alpha,v',v}(\omega) \ge 0$ for \emph{all} $\omega \in \Omega$
	if $v \ge 0$, since $\rho$ is positive. This means
	$M_\alpha(\omega) \ge 0$, hence $M(\omega) \ge 0$, for all
	$\omega \in \Omega$. Thus the representative $M$ constructed
	in the proof is positive if $T$ is positive, and the claim
	is proved.
\end{proof}

It is also possible to relate further order properties of $k$ and $T_k$.
Recall that an operator $T$ from $E$ to $F$ is called \emph{regular} if it
can be written as a difference of positive operators.
The space of all regular operators from $E$ to $F$ is denoted by $\mathscr{L}^r(E,F)$.
If $F$ is Dedekind complete, for example if $F$ is the dual of a Banach lattice, then
every regular operator $T$ from $E$ to $F$ has a modulus $|T| = \sup\{ T, -T \}$, and
$\|T\|_r \coloneqq \|\,|T|\,\|$ makes $\mathscr{L}^r(E,F)$ into a Dedekind complete Banach
lattice~\cite[Theorem~4.74]{Ali06}.

Let $T_k$ be defined via~\eqref{repr:formula}.
We will characterize the kernels $k$ such that $T_k$ is regular.
In fact, we will find an isometric lattice isomorphism between
$\mathscr{L}^r(L^1(\Omega_1;E), L^\infty_{\sigma^\ast}(\Omega_2;F'))$
and an appropriate space of kernels.

For this, we introduce the so-called $p$-tensor product of Banach lattices.
A pair $(H,\chi)$, where $H$ is a Banach lattice and $\chi$ is a positive bilinear
map from $E \times F$ to $H$ such that $\|\chi(u,v)\| = \|u\|\,\|v\|$,
is called a \emph{$p$-tensor product} of
$E$ and $F$ if for every Banach lattice $G$ and for every bounded,
positive\footnote{A bilinear map $\phi\colon E \times F \to G$ is called positive if
$\phi(u,v) \in G_+$ for all $u \in E_+$ and $v \in F_+$.}
bilinear map $\phi$ from $E \times F$ to $G$ there exists a unique bounded linear
operator $T$ from $H$ to $G$ such that the diagram
\begin{center}
\begin{tikzpicture}
	\centering
	\matrix (m) [matrix of math nodes, row sep=2.5em, column sep=4em, text height=1.5ex, text depth=0.25ex]
	{ E \times F & G \\ H & \\ };
	\path[->]
	(m-1-1) edge node[auto] {$\phi$} (m-1-2)
			edge node[auto,swap] {$\chi$} (m-2-1);
	\path[->,dashed]
	(m-2-1) edge node[auto,swap] {$T$} (m-1-2);
\end{tikzpicture}
\end{center}
commutes, and, moreover, this unique map $T$ is positive and satisfies $\|T\| = \|\phi\|$.

It is not hard to check that if $(H_1,\chi_1)$ and $(H_2,\chi_2)$ are $p$-tensor
products of $E$ and $F$, then there exists a (unique) isometric lattice isomorphism
$J$ from $H_1$ to $H_2$ such that $J \circ \chi_1 = \chi_2$.
Moreover, Schlotterbeck~\cite{Schlotterbeck} constructed a $p$-tensor product
for every pair of Banach lattices $E$ and $F$, compare also~\cite{AreDipl}.
Thus there is an essentially unique $p$-tensor product, which we
denote by $E \tilde{\otimes}_p F$, and we will write $u \otimes v$ for $\chi(u,v)$.
Since $T$ is unique, the linear span $E \otimes F$ of the image of $\chi$
is dense in $E \tilde{\otimes}_p F$.

Now fix $G = \mathds{R}$. A positive operator $R$ from $E$ to $F'$ can be
identified with the positive bilinear map $\phi_R$ defined as
$\phi_R(u,v) \coloneqq \apply{Ru}{v}$. By the universal property there
exists a unique element $\psi_R$ in $(E \tilde{\otimes}_p F)'$ such that
$\psi_R(u \otimes v) = \apply{Ru}{v}$ for all $u \in E$ and $v \in F$.
For a regular operator $S$, there exist positive operators $S_1$ and $S_2$
such that $S = S_1 - S_2$, and we define $\psi_S \coloneqq \psi_{S_1} - \psi_{S_2}$.
This definition does not depend on the choice of $S_1$ and $S_2$ and
is consistent with the definition for positive operators.

We claim that $S \mapsto \psi_S$ is an isometric lattice isomorphism from
$\mathscr{L}^r(E,F')$ to $(E \tilde{\otimes}_p F)'$.
It is not hard to see that the mapping is linear and that $\psi_S = 0$ only if $S = 0$.
It follows from the universal property that $S$ is positive if and only
if $\psi_S$ is positive. For $\psi \in (E \tilde{\otimes}_p F)'_+$,
the positive bilinear map $\phi \coloneqq \psi \circ \chi$
defines a positive operator $S$ via $\apply{Su}{v} \coloneqq \phi(u,v)$
such that $\psi = \psi_S$. Since the cone of $(E \tilde{\otimes}_p F)'$ is
generating, this shows that for every $\psi \in (E \tilde{\otimes}_p F)'$ there
exists a regular operator $S$ such that $\psi = \psi_S$.
Hence $S \mapsto \psi_S$ is a lattice isomorphism.
Since for functionals the regular norm and the operator norm coincide,
the universal property shows that
\[
	\|S\|_r = \|\,|S|\,\| = \| \phi_{|S|} \| = \|\psi_{|S|}\| = \|\,|\psi_S|\,\| = \|\psi_S\|,
\]
i.e., that $S \mapsto \psi_S$ is isometric.

Using the universal property, one can show that $\tilde{\otimes}_p$ is
commutative and associative via the natural identifications, for example
by extending the definition of the $p$-tensor product to $n$ factors
and showing that $(E \tilde{\otimes}_p F) \tilde{\otimes}_p G$
and $E \tilde{\otimes}_p (F \tilde{\otimes}_p G)$ are models of
$E \tilde{\otimes}_p F \tilde{\otimes}_p G$.

There is a striking similarity between the definitions of the
projective and the $p$-tensor product. In fact, the definitions
agree apart from the additional positivity conditions.
Thus the following result is not surprising.

\begin{lemma}
	Let $\Omega$ be a measure space and $F$ be a Banach lattice.
	Then $L^1(\Omega;F)$ is isometrically order isomorphic to
	$L^1(\Omega) \tilde{\otimes}_p F$.
\end{lemma}
\begin{proof}
	It suffices to show that $(L^1(\Omega;F),\chi)$ with $\chi(f,v) \coloneqq f(\cdot)v$ has
	the universal property of the $p$-tensor product of $L^1(\Omega)$ and $F$.
	Then $\chi$ is positive and $\|\chi(f,v)\|_{L^1(\Omega;F)} = \|f\|_1 \, \|v\|$.
	Let $G$ be a Banach lattice and $\phi$ be a bounded, positive bilinear
	functional from $L^1(\Omega) \times F$ to $G$.

	If there exists a bounded linear operator $T$ as in the universal property,
	it satisfies
	\begin{equation}\label{simple:uniq}
		T\Bigl(\sum_{i=1}^n \setone_{A_i} v_i\Bigr) = \sum_{i=1}^n \phi(f_i,v_i)
	\end{equation}
	for all choices of measurable sets $A_i \subset \Omega$ and vectors $v_i \in F$.
	Since the simple functions are dense in $L^1(\Omega;F)$, this shows
	uniqueness of $T$, if it exists.
	Moreover, if a bounded $T$ satisfies~\eqref{simple:uniq},
	then $T \circ \chi = \phi$, since this holds on a dense subspace
	and $\chi$ and $\phi$ are continuous.
	Furthermore, if $T$ satisfies~\eqref{simple:uniq}, then
	\[
		\|T\| \ge \sup\bigl\{ T(\chi(f,v)) : \|f\|_1 \le 1, \; \|v\| \le 1 \bigr\}
			= \sup\bigl\{ \phi(f,v) : \|f\|_1 \le 1, \; \|v\| \le 1 \bigr\}
			= \|\phi\|.
	\]

	We still have to show that there exists a bounded, positive operator $T$
	that satisfies~\eqref{simple:uniq} and $\|T\| \le \|\phi\|$.
	Define $T$ on the simple functions by~\eqref{simple:uniq}.
	For $v_i \in F$ and disjoint, measurable sets $A_i \subset \Omega$ of positive measure,
	let $g \coloneqq \sum_{i=1}^n \setone_{A_i} v_i$. Then
	\[
		\| Tg \|
			= \Bigl\| \sum_{i=1}^n \phi(\setone_{A_i}, v_i) \Bigr\|
			\le \|\phi\| \sum_{i=1}^n \|\setone_{A_i}\| \, \|v_i\|
			= \|\phi\| \, \| g \|_{L^1(\Omega;F)}.
	\]
	Since such $g$ are dense in $L^1(\Omega;F)$,
	this shows that $T$ has a continuous extension to $L^1(\Omega;F)$ such that
	$\|T\| \le \|\phi\|$.
	From~\eqref{simple:uniq} we obtain that $Tg \in G_+$ if $g \ge 0$, i.e.,
	if $v_i \in F_+$ for all $i$. Since such functions $g$ are dense in
	$L^1_+(\Omega;F)$ and $G_+$ is closed in $G$, $T$ is positive.

	We have checked all conditions in the universal property.
	Hence $L^1(\Omega;F)$ and $L^1(\Omega) \tilde{\otimes}_p F$ are isometrically
	order isomorphic by the uniqueness of the $p$-tensor product.
\end{proof}

The lemma shows in particular that $L^1(\Omega) \tilde{\otimes}_p L^1(\Omega)$
is isomorphic to $L^1(\Omega \times \Omega)$, compare~\eqref{L1t} and Lemma~\ref{L1E}.

Now we are in the position to prove the announced theorem about the correspondence of the
regularity of the kernel and the operator.
\begin{thm}\label{regkernel}
	Let $\Omega_1$ and $\Omega_2$ be complete, strictly localizable measure spaces, and
	let $E$ and $F$ be Banach lattices.
	Then $k \mapsto T_k$ as defined in~\eqref{repr:formula} is an isometric lattice isomorphism from
	$L^\infty_{\sigma^\ast}(\Omega_1 \times \Omega_2; \mathscr{L}^r(E,F'))$
	onto $\mathscr{L}^r(L^1(\Omega_1;E), L^\infty_{\sigma^\ast}(\Omega_2;F'))$.
\end{thm}
\begin{proof}
	Recall that $L^\infty_{\sigma^\ast}(\Omega;F')$ and $L^1(\Omega;F)'$
	are isometrically isomorphic.
	Using the properties of the $p$-tensor product, we can proceed as in the proof of
	Theorem~\ref{repr}. We obtain that
	\begin{align*}
		L^\infty_{\sigma^\ast}(\Omega_1 \times \Omega_2; \mathscr{L}^r(E,F'))
			& \cong L^1(\Omega_1 \times \Omega_2; E \tilde{\otimes}_p F)'
			\cong (L^1(\Omega_1;E) \tilde{\otimes}_p L^1(\Omega_2;F))' \\
			& \cong \mathscr{L}^r(L^1(\Omega;E), L^\infty_{\sigma^\ast}(\Omega_2;F'))
	\end{align*}
	is a chain of isometric lattice isomorphisms. The composition is given by $k \mapsto T_k$.
	This proves the claim.
\end{proof}

The theorem shows that the regular operators are precisely those
whose kernel takes values in the regular operators and is bounded
even with respect to the regular norm.
In particular, not even every kernel in $\mathrm{C}([0,1]^2; \mathscr{L}(E,F'))$
such that $k(x,y) \in \mathscr{L}^r(E,F')$ for all $(x,y) \in [0,1]^2$
defines a regular operator $T_k$. The next example shows how one
can construct such a kernel.

\begin{exa}
	Let $E \coloneqq F' \coloneqq H \coloneqq L^2(\mathds{T})$ with the natural ordering,
	where $\mathds{T}$ denotes the one dimensional torus.
	The space $\mathscr{L}^r(H)$ is not closed in $\mathscr{L}(H)$, see for
	example~\cite[Counterexample~3.7]{Are81}. Hence there exist regular
	operators $S_n$ such that $\|S_n\| \to 0$ and $\|S_n\|_r \to \infty$
	as $n \to \infty$. By linear interpolation we find a continuous function $f$
	from $[0,1]$ to $\mathscr{L}(H)$ taking values in $\mathscr{L}^r(H)$
	such that $\|f(x)\|_r \to \infty$ as $x \to 0$.
	Then $k(x,y) \coloneqq f(x)$ defines a function in $\mathrm{C}([0,1]^2;\mathscr{L}(H))$,
	$k(x,y) \in \mathscr{L}^r(H)$ for all $(x,y) \in [0,1]^2$,
	but $T_k$ is by Theorem~\ref{regkernel} not regular.
\end{exa}

\section{$L^p$-Representations}\label{lpsec}
In Theorem~\ref{repr} we have characterized the operators between spaces of vector-valued
functions that come from bounded kernels. In the scalar case this reduces to Theorem~\ref{Repr}.
For scalar-valued functions, there is one other class of kernels for which it is fairly simple
to describe the corresponding operators. In fact, every Hilbert-Schmidt operator from
$L^2(\Omega_1)$ to $L^2(\Omega_2)$ corresponds to a kernel $k \in L^2(\Omega_1 \times \Omega_2)$
via formula~\eqref{scalarrepform}.
This result can easily be generalized to the vector-valued case, too.

Here we denote the space of all Hilbert-Schmidt operators from a Hilbert space $H_1$
to another Hilbert space $H_2$ by $\mathscr{L}_2(H_1, H_2)$.
It is well-known that $\mathscr{L}_2(H_1,H_2)$ itself is again a Hilbert space.

\begin{thm}\label{rep2}
	Let $\Omega_1$ and $\Omega_2$ be $\sigma$-finite measure spaces,
	and let $E$ and $F$ be Hilbert spaces. Then the mapping $k\mapsto T_k$
	defined as in~\eqref{repr:formula} is an isometric isomorphism from
	$L^2(\Omega_1\times \Omega_2; \mathscr{L}_2(E,F))$ onto
	$\mathscr{L}_2(L^2(\Omega_1;E); L^2(\Omega_2;F))$.
\end{thm}

\begin{proof}
	We make use of standard properties of the usual complete Hilbert space tensor
	product $\tilde{\otimes}_\sigma$ (cf.~\cite[\S 2.6]{KR97} for details) and deduce that 
	\begin{align*}
		\mathscr{L}_2(L^2(\Omega_1; E), L^2(\Omega_2; F))
			& \cong \bigl(L^2(\Omega_1) \tilde{\otimes}_\sigma E\bigr) \tilde{\otimes}_\sigma \bigl(L^2(\Omega_2) \tilde{\otimes}_\sigma F\bigr)
			\cong L^2(\Omega_1) \tilde{\otimes}_\sigma L^2(\Omega_2) \tilde{\otimes}_\sigma E \tilde{\otimes}_\sigma F \\
			& \cong L^2(\Omega_1 \times \Omega_2) \tilde{\otimes}_\sigma \mathscr{L}_2(E,F)
			\cong L^2(\Omega_1 \times \Omega_2; \mathscr{L}_2(E,F))
	\end{align*}
	is a chain of isometric isomorphisms.
\end{proof}

In particular, the square of the Hilbert-Schmidt norm of an operator $T_k$ equals
\[
	\|T_k\|_{\mathscr{L}_2}^2 = \int_{\Omega_1\times\Omega_2} \|k(x,y)\|_{\mathscr{L}_2(E,F)}^2 \dx(x,y).
\]

\begin{rem}
	Theorem~\ref{rep2} shows that, in contrast to the scalar-valued case, not every
	operator mapping $L^1(0,1;H)$ to $L^\infty(0,1;H)$, $H$ a Hilbert space, is
	a Hilbert-Schmidt operator. This is not very surprising, though, since such operators
	are in general not even compact if $H$ is infinite-dimensional.
\end{rem}

It is natural to ask whether we can find characterizations of operators associated
with kernels in $L^p(\Omega_1 \times \Omega_2; \mathcal{A}(E,F'))$, where
$\mathcal{A}(E,F')$ is a suitable subspace of $\mathscr{L}(E,F')$, for example
a space of Schatten class operators if $E$ and $F$ are Hilbert spaces.

In the scalar case, the situation is well-understood and leads to the concept
of Hille-Tamarkin operators. It is for example possible to describe the spectrum
of such operators rather precisely due to a celebrated result by
Johnson-K\"onig-Maurey-Retherford~\cite{JKMR79}. More on this topic can
be found in~\cite[\S 6.4]{Pietsch07}.

However, we run into severe difficulties when we try to attack this question
with the above techniques involving tensor products, even in the scalar case.
Although there are natural candidates for suitable tensor products in the
$L^p$-setting, e.g.~$\tilde{\otimes}_m$ and $\tilde{\otimes}_l$
in~\cite{schsch75} and $\tilde{\otimes}_{\Delta_p}$
in~\cite[Chapter~7]{defflo93}, these are just not as
well-behaved as the projective or Hilbert space tensor product.
For example, in general $L^p(\Omega)\tilde{\otimes} H \not\cong H\tilde{\otimes} L^p(\Omega)$
for any of these tensor products, i.e., $\tilde{\otimes}$ is not commutative.

\section*{Acknowledgments}
The authors wish to thank Prof.\ Wolfgang Arendt and Prof.\ Ulf Schlotterbeck
for valuable discussions and motivating help concerning the work
on this article, and in particular for profound information about the
$p$-tensor product.

\bibliographystyle{amsplain}
\bibliography{dunfordpettis}

\providecommand{\bysame}{\leavevmode\hbox to3em{\hrulefill}\thinspace}
\providecommand{\MR}{\relax\ifhmode\unskip\space\fi MR }
\providecommand{\MRhref}[2]{%
  \href{http://www.ams.org/mathscinet-getitem?mr=#1}{#2}
}
\providecommand{\href}[2]{#2}
\begin{thebibliography}{10}

\bibitem{Ali06}
C.D. Aliprantis and O.~Burkinshaw, \emph{{Positive Operators}}, Springer,
  Dordrecht, 2006, Reprint of the 1985 original.

\bibitem{Amann01b}
H.~Amann, \emph{{Elliptic operators with infinite-dimensional state spaces}},
  J. Evol. Equ. \textbf{1} (2001), no.~2, 143--188.

\bibitem{and83}
K.T. Andrews, \emph{{The Radon-Nikodym property for spaces of operators}}, J.
  Lond. Math. Soc., II. Ser. \textbf{28} (1983), 113--122.

\bibitem{AreDipl}
W.~Arendt, \emph{{Fortsetzungsprobleme f{\"u}r Operatoren zwischen
  Banachr{\"a}umen und Banachverb{\"a}nden}}, Diplomarbeit, Universit{\"a}t
  T{\"u}bingen, 1975.

\bibitem{ArendtSurvey}
\bysame, \emph{{Semigroups and evolution equations: functional calculus,
  regularity and kernel estimates}}, {Evolutionary equations. Vol. I}, Handb.
  Differ. Equ., North-Holland, Amsterdam, 2004, pp.~1--85.

\bibitem{ABHN01}
W.~Arendt, C.J.K. Batty, M.~Hieber, and F.~Neubrander, \emph{Vector-valued
  {L}aplace transforms and {C}auchy problems}, Monographs in Mathematics,
  vol.~96, Birkh\"auser Verlag, Basel, 2001.

\bibitem{AreBuk94}
W.~Arendt and A.V. Bukhvalov, \emph{{Integral representations of resolvents and
  semigroups}}, Forum Math. \textbf{6} (1994), 111--135.

\bibitem{AreTho05}
W.~Arendt and S.~Thomaschewski, \emph{{Local operators and forms}}, Positivity
  \textbf{9} (2005), no.~3, 357--367.

\bibitem{Are81}
Wolfgang Arendt, \emph{{On the o-spectrum of regular operators and the spectrum
  of measures}}, Math. Z. \textbf{178} (1981), no.~2, 271--287.

\bibitem{AO97}
J.A. Ausekle and E.F. Oja, \emph{{Pitt's Theorem for the Lorentz and Orlicz
  Sequence Spaces}}, Mathematical Notes \textbf{61} (1997), no.~1, {16--21}.

\bibitem{Burke93}
M.R. Burke, \emph{Liftings for noncomplete probability spaces}, Papers on
  general topology and applications ({M}adison, {WI}, 1991), vol. 704, New York
  Acad. Sci., 1993, pp.~34--37.

\bibitem{Davies90}
E.B. Davies, \emph{{Heat kernels and spectral theory}}, Cambridge Tracts in
  Mathematics, vol.~92, Cambridge University Press, 1990.

\bibitem{defflo93}
A.~Defant and K.~Floret, \emph{{Tensor Norms and Operator Ideals}},
  North-Holland Mathematics Studies, vol. 176, North-Holland, Amsterdam, 1993.

\bibitem{DHP07}
R.~Denk, M.~Hieber, and J.~Pr{\"u}ss, \emph{{Optimal {$L\sp p$}-{$L\sp
  q$}-estimates for parabolic boundary value problems with inhomogeneous
  data}}, Math. Z. \textbf{257} (2007), no.~1, 193--224.

\bibitem{diefouswa08}
J.~Diestel, J.H. Fourie, and J.~Swart, \emph{{The metric theory of tensor
  products. Grothendieck's r{\'e}sum{\'e} revisited}}, Am. Math. Soc.,
  Providence, RI, 2008.

\bibitem{DunPet40}
N.~Dunford and B.J. Pettis, \emph{{Linear operations on summable functions}},
  Trans. Am. Math. Soc. \textbf{47} (1940), 323--392.

\bibitem{Flo74}
K.~Floret, \emph{{Der Satz von Dunford-Pettis und die Darstellung von Massen
  mit Werten in lokalkonvexen R{\"a}umen}}, Mathematische Annalen \textbf{208}
  (1974), 203--212.

\bibitem{Frem2}
D.H. Fremlin, \emph{{Broad Foundations}}, Measure Theory, vol.~2, Colchester,
  Torres Fremlin, 2001.

\bibitem{Frem3}
\bysame, \emph{{Measure Algebras}}, Measure Theory, vol.~3, Colchester, Torres
  Fremlin, 2002.

\bibitem{Frem41}
\bysame, \emph{{Topological Measure Spaces}}, Measure Theory, vol. 4, part I,
  Colchester, Torres Fremlin, 2003.

\bibitem{Gel38}
I.~Gelfand, \emph{Abstrakte funktionen und lineare operatoren}, Rec. Math.
  Moscou, n. Ser. \textbf{4} (1938), 235--284.

\bibitem{Tulcea69}
A.~Ionescu~Tulcea and C.~Ionescu~Tulcea, \emph{{Topics in the Theory of
  Lifting}}, Ergebnisse der Mathematik und ihrer Grenzgebiete, Band 48,
  Springer-Verlag New York Inc., New York, 1969.

\bibitem{JKMR79}
W.B. Johnson, H.~K{\"o}nig, B.~Maurey, and J.R. Retherford, \emph{{Eigenvalues
  of {$p$}-summing and {$l\sb{p}$}-type operators in Banach spaces}}, J. Funct.
  Anal. \textbf{32} (1979), no.~3, 353--380.

\bibitem{KR97}
R.V. Kadison and J.R. Ringrose, \emph{{Fundamentals of the theory of operator
  algebras. Vol. I}}, Graduate Studies in Mathematics, vol.~15, American
  Mathematical Society, Providence, RI, 1997, Elementary theory, Reprint of the
  1983 original.

\bibitem{KanVul37}
L.V. Kantorovich and B.Z. Vulikh, \emph{{Sur la repr{\'e}sentation des
  op{\'e}rations lin{\'e}aires}}, Compos. Math. \textbf{5} (1937), 119--165.

\bibitem{Koethe79}
G.~K{\"o}the, \emph{{Topological Vector Spaces II}}, Springer, 1979.

\bibitem{Kuch05}
P.~Kuchment, \emph{{Quantum graphs II. Some spectral properties of quantum and
  combinatorial graphs}}, J. Phys. A \textbf{38} (2005), no.~22, 4887--4900.

\bibitem{LinTza77}
J.~Lindenstrauss and L.~Tzafriri, \emph{{Classical Banach spaces I}},
  Ergebnisse der Mathematik und ihrer Grenzgebiete, vol.~92, Springer-Verlag,
  1977.

\bibitem{Pietsch07}
A.~Pietsch, \emph{{History of Banach spaces and linear operators}},
  Birkh{\"a}user Boston Inc., Boston, MA, 2007.

\bibitem{Ros69}
H.P. Rosenthal, \emph{{On quasi-complemented subspaces of Banach spaces, with
  an appendix on compactness of operators from {$L\sp{p}\,(\mu )$} to
  {$L\sp{r}\,(\nu )$}}}, J. Functional Analysis \textbf{4} (1969), 176--214.

\bibitem{schsch75}
H.H. Schaefer and U.~Schlotterbeck, \emph{{On the approximation of kernel
  operators by operators of finite rank}}, J. Approximation Theory (1975),
  33--39.

\bibitem{Schaefer99}
H.H. Schaefer and M.P.H. Wolff, \emph{{Topological Vector Spaces}}, Springer,
  1999.

\bibitem{Schlotterbeck}
U.~Schlotterbeck, \emph{{Tensorprodukte von Banachverb{\"a}nden und positive
  Operatoren}}, Habilitationsschrift, Universit{\"a}t T{\"u}bingen, 1977.

\bibitem{vBL05}
J.~von Below and J.A. Lubary, \emph{{The eigenvalues of the Laplacian on
  locally finite networks}}, Results Math. \textbf{47} (2005), 199--225.

\end{thebibliography}

\end{document}